\documentclass[11pt]{amsart}
\usepackage{amsthm,amsmath,amssymb,enumerate}
\title[]{The Hessian of quantized Ding functionals and its asymptotic behavior}
\author[]{Ryosuke Takahashi}

\address{Mathematical Institute, Tohoku University, 6-3, Aoba, Aramaki, Aoba-ku, Sendai, 980-8578, Japan}

\email{ryosuke.takahashi.a7@tohoku.ac.jp}

\keywords{Ding functional, quantization, K\"ahler-Einstein metric}

\thanks{This work was supported by Grant-in-Aid for JSPS Fellows Number 16J01211.}
\subjclass[2010]{53C25}
\makeatletter

\@addtoreset{equation}{section}
\makeatother

{
\theoremstyle{definition}

\newtheorem*{acknowledgements}{Acknowledgements}
}

{
\theoremstyle{plain}
\newtheorem{theorem}{Theorem}[section]

\newtheorem{lemma}{Lemma}[section]
\newtheorem{corollary}{Corollary}[section]
}

{
\theoremstyle{remark}
\newtheorem{remark}{Remark}[section]

}
\begin{document}
%=========Abstract===================================================
\begin{abstract}
We compute the Hessian of quantized Ding functionals and give an elementary proof for the convexity of quantized Ding functionals along Bergman geodesics from the view point of projective geometry. We study also the asymptotic behavior of the Hessian using the Berezin-Toeplitz quantization.
\end{abstract}
\maketitle
%=========Section 1===================================================
\section{Introduction} 
Let $X$ be an $n$-dimensional Fano manifold and $k$ a large integer such that $-kK_X$ is very ample. Let ${\mathcal H}(-K_X)$ be the space of smooth fiber metrics $\phi$ on $-K_X$ with positive curvature
$\omega_{\phi}:=(\sqrt{-1}/2 \pi) \partial \bar{\partial} \phi$
and ${\mathcal B}_k$ the space of hermitian forms on $H^0(X,-kK_X)$. For a metric $\phi \in {\mathcal H}(-K_X)$, we denote the Monge-Amp\`ere volume form by
\[
MA(\phi):=\frac{\omega_{\phi}^n}{n!},
\]
and the canonical volume form by
\[
e^{-\phi}:=(\sqrt{-1})^{n^2} \left|\frac{\partial}{\partial z_1} \wedge \cdots \wedge \frac{\partial}{\partial z_n} \right|_{\phi}^2 dz_1 \wedge \cdots \wedge dz_n \wedge d\bar{z}_1 \wedge \cdots \wedge d\bar{z}_n.
\]
This expression is readily verified to be independent of the local holomorphic coordinates $(z_1, \ldots, z_n)$ and hence defines a volume form on $X$. We normalize $e^{-\phi}$ to be a probability measure
\[
\mu_{\phi}:=\frac{e^{-\phi}}{\int_X e^{-\phi}}.
\]
The space ${\mathcal H}(-K_X)$ admits a natural Riemannian metric, called {\it Mabuchi metric}, defined by the $L^2$-norm of a tangent vector $f \in C^{\infty}(X,{\mathbb R})$ at $\phi$: $|f|_{\phi}^2:=\int_X f^2 MA(\phi)$. 

Recall that {\it Ding functionals} ${\mathcal D}$ on the space of infinite dimensional Riemannian manifold ${\mathcal H}(-K_X)$ is uniquely characterized (modulo an additive constant) by the property
\[
\nabla {\mathcal D}_{|\phi}=\frac{\mu_{\phi}}{MA(\phi)}-\frac{n!}{(-K_X)^n},
\]
where $\nabla$ denotes the gradient and $(-K_X)^n$ is the top intersection number. The Ding functionals are important in the study of K\"ahler-Einstein metrics. For instance, Donaldson \cite{Don15} recently gave a ``moment map'' interpretation of the Ding functional. More precisely, he showed that the ratio of volumes $\mu_{\phi}/MA(\phi)-n!/(-K_X)^n$ arises as the moment map for a suitable infinite dimensional symplectic manifold and the Ding functional can be viewed as the Kempf-Ness function. This interpretation provides us the direct link between the existence problem of K\"ahler-Einstein metrics and stability in Geometric Invariant Theory.

On the other hand, quantization of the Ding functionals is also studied: given $\phi \in {\mathcal H}(-K_X)$, we define a hermitian form $Hilb_k(\phi) \in {\mathcal B}_k$ by
\[
\|s\|_{Hilb_k(\phi)}^2:=\int_X |s|_{k \phi}^2 MA(\phi).
\]
Conversely, for a given $H \in {\mathcal B}_k$, we define a metric $FS_k(H) \in {\mathcal H}(-K_X)$ by
\[
FS_k(H):= \frac{1}{k} \log \left( k^{-n} \sup_{s \in H^0(X,-kK_X) \backslash \{0\}} \frac{|s|^2}{H(s,s)} \right).
\]
In what follows, we fix some $H_0 \in {\mathcal B}_k$ and a reference $H_0$-ONB $\underline{s}_0:=(s_1, \ldots, s_{N_k})$, which defines an embedding $\iota_{\underline{s}_0} \colon X \hookrightarrow {\mathbb CP}^{N_k-1}$, where $N_k:=\dim H^0(X,-kK_X)$. For $g \in GL(N_k, {\mathbb C})$, let $H_g \in {\mathcal B}_k$ be a hermitian form such that $\underline{s}_0 \cdot g$ is an $H_g$-orthonormal basis. Then the map $g \mapsto H_g$ gives an isomorphism $GL(N_k, {\mathbb C})/U(N_k) \simeq {\mathcal B}_k$, and the tangent space of ${\mathcal B}_k$ can be identified with $\sqrt{-1}{\mathfrak u}(N_k)$. Thus the space ${\mathcal B}_k$ admits a natural Riemannian structure defined by the Killing form ${\rm tr}(AB)$ ($A, B \in \sqrt{-1}{\mathfrak u}(N_k)$) at each tangent space.

Write $M \colon {\mathbb CP}^{N_k-1} \to \sqrt{-1}{\mathfrak u}(N_k)$ for the map:
\[
M([Z_1; \cdots ; Z_{N_k}]):=\left( \frac{\overline{Z}_{\alpha} Z_{\beta}}{\sum_i |Z_i|^2} \right)_{\alpha \beta}.
\]
We define the {\it center of mass} $\underline{M}(g)$ by the formula
\[
\underline{M}(g):=\int_X (M \circ g) \mu_{FS_k(H_g)},
\]
where we identify $X$ with $\iota_{\underline{s}_0} (X)$ and the measure $\mu_{FS_k(H_g)}$ with its push forward $(\iota_{\underline{s}_0})_{\ast} \mu_{FS_k(H_g)}$.
{\it Quantized Ding functionals} ${\mathcal D}^{(k)}$ on the space of finite dimensional Riemannian manifold ${\mathcal B}_k$ is uniquely characterized (modulo an additive constant) by the property
\[
\nabla {\mathcal D}^{(k)} \hbox{}_{|H_g}=k^{-1} \left(\underline{M}(g)-\frac{\rm Id}{N_k} \right).
\]
There is also a finite dimensional moment-map picture for this setting: the space of all bases ${\mathcal Z}^{(k)} \simeq GL(N_k)$ is equipped with a K\"ahler structure induced from the Berndtsson metric and $U(N_k)$ acts on ${\mathcal Z}^{(k)}$ isometrically with a moment map which is essentially $\underline{M}(g)-{\rm Id}/N_k$. Critical points of the quantized Ding functionals $\underline{M}(g)={\rm Id}/N_k$ are called {\it anti-canonically balanced metrics}. There is a strong connection between the existence problem of anti-canonically balanced metrics and K\"ahler-Einstein metrics (cf. \cite[Theorem 7.1]{BBGZ13}).

In this paper, we study the Hessian of quantized Ding functionals and its asymptotic behavior as raising exponent $k \to \infty$. We first give a formula for the Hessian of the quantized Ding functional $\nabla^2 {\mathcal D}^{(k)}$:
\begin{theorem} \label{fhd}
The Hessian of the quantized Ding functional is computed by
\begin{eqnarray*}
\nabla^2 {\mathcal D}^{(k)} \hbox{}_{|H_0}(A,B)&=&k^{-1} \int_X {\rm Re}(\xi_A, \xi_B)_{FS} \mu_{FS_k(H_0)}-k^{-2} \int_XH(A)H(B) \mu_{FS_k(H_0)} \\
&+&k^{-2} \int_X H(A) \mu_{FS_k(H_0)} \cdot \int_X H(B) \mu_{FS_k(H_0)},
\end{eqnarray*}
where $\xi_A$ denotes the holomorphic vector filed on ${\mathbb CP}^{N_k-1}$ corresponding to $A$, and $H(A)$ is the Hamiltonian for the Killing vector $J\xi_A^{\mathbb R}$, $(\cdot,\cdot)_{FS}$ is the Fubini-Study inner product on tangent vectors.
\end{theorem}
As a corollary, we will show the following:
\begin{corollary} \label{coq}
We have $\nabla^2 {\mathcal D}^{(k)} \hbox{}_{|H_0}(A,A) \geq 0$ for any $A \in \sqrt{-1}{\mathfrak u}(N_k)$, and the equality holds if and only if $A \in {\rm Lie}({\rm Aut}(X,-kK_X))$, where ${\rm Aut}(X,-kK_X)$ denotes the group of holomorphic automorphisms of the pair $(X,-kK_X)$, embedded into $GL(N_k,{\mathbb C})$ by means of the reference basis $\underline{s}_0$.
\end{corollary}
Although Corollary \ref{coq} is a direct consequence of Berndtsson's convexity theorem \cite[Theorem 2.4]{Ber09} (see also \cite[Lemma 7.2]{BBGZ13}), our proof is completely independent, based on the viewpoint of projective geometry, and somewhat elementary.

Next, we fix a reference metric $\phi_0 \in {\mathcal H}(-K_X)$ and set $H_0:=Hilb_k(\phi_0)$. For $f \in C^{\infty}(X, {\mathbb R})$, we associate with the derivative of the Hilbert map in the direction $f$:
\[
Q_{f,k}:=\frac{d}{dt} Hilb_k(\phi_0-tf)_{|t=0}= \left( \int_X (kf-\Delta_{\phi_0} f)(s_{\alpha},s_{\beta})_{k\phi_0} MA(\phi_0) \right)_{\alpha \beta},
\]
where $\Delta_{\phi_0}$ denotes the (negative) $\bar{\partial}$-Laplacian with respect to $\phi_0$, and in the last equality, we identified $Q_{f,k}$ with a hermitian matrix by means of the reference basis $\underline{s}_0$. With this operation, we can connect $\nabla^2 {\mathcal D}^{(k)}$ to $\nabla^2 {\mathcal D}$ as follows:
\begin{theorem} \label{coh}
For any functions $f,g \in C^{\infty}(X,{\mathbb R})$, we have the convergence of the Hessian
\begin{equation} \label{che}
\nabla^2 {\mathcal D}^{(k)} \hbox{}_{|H_0} (Q_{f,k},Q_{g,k}) \to \nabla^2 {\mathcal D} \hbox{}_{|\phi_0} (f,g)
\end{equation}
as $k \to \infty$. In particular, $\lim_{k \to \infty} \nabla^2{\mathcal D}^{(k)} \hbox{}_{|H_0} (Q_{f,k},Q_{f,k})=0$ implies that the condition characterizing degeneracy $\nabla^2 {\mathcal D}\hbox{}_{|\phi_0}(f,f)=0$ follows. Finally, the above convergence is uniform when $f,g$ vary in a subset of $C^{\infty}(X,{\mathbb R})$ which is compact for the $C^{\infty}$-topology. It is also uniform for $\phi_0$ as long as $\phi_0$ stays a compact set in the $C^{\infty}$-toplogy.
\end{theorem}
This is an analogue of Berndtsson's result \cite[Theorem 4.1]{Ber09}, but quantization schemes are different. Moreover, his argument is based on Hodge theory, whereas an important technical tool we use in our proofs is Berezin-Toeplitz quantization provided by Ma-Marinescu \cite{MM12}. We should mention as well that J. Fine \cite{Fin12} studied the quantization of the Lichnerowicz operator on general polarized manifolds. Our method follows a strategy discovered by him.
\begin{acknowledgements}
The author would like to express his gratitude to his advisor Professor Shigetoshi Bando for useful discussions on this article. This research is supported by Grant-in-Aid for JSPS Fellows Number 16J01211.
\end{acknowledgements}
\section{Foundations}
\subsection{Functionals on the space of metrics}
We have a quick review on several functionals over the space of metrics ${\mathcal H}(-K_X)$ or ${\mathcal B}_k$ which play a central role in the study of K\"ahler-Einstein metrics. The standard reference for this section is \cite{BBGZ13}. We fix a reference metric $\phi_0 \in {\mathcal H}(-K_X)$. We define the {\it Monge-Amp\`ere energy} by
\[
{\mathcal E}(\phi):=\frac{1}{(n+1)(-K_X)^n} \sum_{i=1}^n \int_X (\phi-\phi_0) \omega_{\phi}^{n-i} \wedge \omega_{\phi_0}^i,
\]
and the Ding functional ${\mathcal D}$ by
\[
{\mathcal D}(\phi):=-{\mathcal E}(\phi)+{\mathcal L}(\phi), \;\;\; {\mathcal L}(\phi):=-\log \int_X e^{-\phi}.
\]
Let $h_{\phi}$ be the Ricci potential of $\omega_{\phi}$:
\[
h_{\phi}:=\log \frac{\mu_{\phi}}{MA(\phi)}.
\]
A direct computation implies that the derivative of ${\mathcal D}$ along a smooth curve $\phi_t$ in ${\mathcal H}(-K_X)$ is
\begin{eqnarray*}
\frac{d^2}{dt^2}{\mathcal D}(\phi_t)&=&- \int_X ( \ddot{\phi}_t-|\partial \dot{\phi_t}|_{\phi_t}^2) \left( \frac{n!}{(-K_X)^n} - e^{h_{\phi_t}} \right) MA(\phi) \\
&+& \int_X  |\partial \dot{\phi}_t|_{\phi_t}^2 \mu_{\phi_t} - \int_X \dot{\phi}_t^2 \mu_{\phi_t} + \left( \int_X \dot{\phi}_t \mu_{\phi_t} \right)^2.
\end{eqnarray*}
Hence the Hessian of ${\mathcal D}$ is
\begin{equation}
\nabla^2 {\mathcal D} \hbox{}_{|\phi} (f,g)=\int_X  {\rm Re}(\partial f, \partial g)_{\phi} \mu_{\phi} - \int_X fg \mu_{\phi} + \int_X f \mu_{\phi} \cdot \int_X g \mu_{\phi}.
\end{equation}
\begin{remark} \label{frk}
We find that the Hessian $\nabla^2 {\mathcal D}$ is non-negative by the modified Poincar\'e inequality on Fano manifolds (for instance, see \cite[Corollary 2.1]{TZ00}).
\end{remark}
Set $H_0:=Hilb_k(\phi_0)$. we also define the {\it quantized Monge-Amp\`ere energy} by
\[
{\mathcal E}^{(k)}(H):=-\frac{1}{kN_k} \log \det (H \cdot H_0^{-1}),
\]
and the {\it quantized Ding functional} by
\[
{\mathcal D}^{(k)}(H):=-{\mathcal E}^{(k)}(H)+({\mathcal L} \circ FS_k)(H).
\]
\subsection{Berezin-Toeplitz quantization}
The key technical result that we use in the proof of Theorem \ref{coh} is the asymptotic expansion of the Bergman function and their generalizations. For $\phi_0 \in {\mathcal H}(-K_X)$, the {\it Bergman function} $\rho_k(\phi_0) \colon X \to {\mathbb R}$ is defined by
\[
\rho_k (\phi_0):=\sum_i |s_i|_{k \phi_0}^2,
\]
where $(s_i)$ is a $Hilb_k(\phi_0)$-ONB of $H^0(X,-kK_X)$. The central result for the Bergman function concerns the large $k$ asymptotic of $\rho_k(\phi_0)$, obtained by Bouche \cite{Bou90}, Catlin \cite{Cat99}, Tian \cite{Tia90} and Zelditch \cite{Zel98}:
\begin{theorem} \label{abf}
We have the following asymptotic expansion of the Bergman function:
\[
\rho_k (\phi_0) = b_0 k^n + b_1 k^{n-1} + b_2 k^{n-2} + \cdots,
\]
where each coefficient $b_i$ can be written as a polynomial in the Riemannian curvature ${\rm Riem}(\omega_{\phi_0})$, their derivatives and contractions with respect to $\omega_{\phi_0}$. In particular,
\[
b_0=1, \;\;\; b_1= \frac{1}{2} S_{\phi_0},
\]
where $S_{\phi_0}$ is the scalar curvature of $\omega_{\phi_0}$. The above expansion is uniform as long as $\phi_0$ stays in a compact set in the $C^{\infty}$-topology. More precisely, for any integer $p$ and $l$, there exists a constant $C_{p, l}$ such that
\[
\left\| \rho_k (\phi_0) - \sum_{i=0}^p b_i k^{n-i} \right\|_{C^l} < C_{p,l} \cdot k^{n-p-1}.
\]
We can take the constant $C_{p, l}$ independently of $\phi_0$ as long as $\phi_0$ stays in a compact set in the $C^{\infty}$-topology.
\end{theorem}
Another important technical tool in our proofs is provided by the Berezin-Toeplitz quantization \cite{MM12}. For $f \in C^{\infty}(X, {\mathbb R})$, the Berezin-Toeplitz operator $T_{f,k}$ is a sequence of linear operators
\[
T_{f,k} \colon H^0(X,-kK_X) \to H^0(X,-kK_X)
\]
defined as two steps: first multiply a given section by $f$, then project to the space of holomorphic sections using the $L^2$-inner product $Hilb_k(\phi_0)$. Using the $Hilb_k(\phi_0)$-ONB $(s_i)$, we obtain the explicit description of the kernel:
\[
\widetilde{K}_{f,k}(x,y)=\sum_{\alpha, \beta} \int_X f(z) (s_{\alpha},s_{\beta})_{k \phi_0}(z) s_{\beta}(y) \otimes s_{\alpha}^{\ast}(x) MA(\phi_0)(z).
\]
If we restrict $\widetilde{K}_{f,k}$ to the diagonal, we have
\[
K_{f,k}(x):=\widetilde{K}_{f,k}(x,x)=\sum_{\alpha, \beta} \int_X f(z) (s_{\alpha}, s_{\beta})_{k \phi_0}(z) (s_{\beta}, s_{\alpha})_{k \phi_0}(x) MA(\phi_0)(z).
\]
\begin{theorem}[\cite{MM12}] \label{afp}
We have the following asymptotic expansion:
\[
\rho_k (\phi_0) = b_{f,0} k^n + b_{f,1} k^{n-1} + \cdots
\]
for smooth functions $b_{f,j}$. Moreover, there are the following formula for coefficients:
\begin{eqnarray*}
b_{f,0}&=&f,\\
b_{f,1}&=&\frac{1}{2}S_{\phi_0} f+\Delta_{\phi_0} f.
\end{eqnarray*}
The expansion is uniform in $f$ varying in a subset of $C^{\infty}(X,{\mathbb R})$ which is compact for the $C^{\infty}$-topology. It is also uniform for $\phi_0$ as long as $\phi_0$ stays a compact set in the $C^{\infty}$-toplogy.
\end{theorem}
For $f,g \in C^{\infty}(X,{\mathbb R})$, we also use the kernel of the composition $T_{f,k} \circ T_{g,k}$:
\[
\widetilde{K}_{f,g,k}(x,y)=\int_X \widetilde{K}_{f,k}(x,z) \widetilde{K}_{g,k}(z,y) MA(\phi_0)(z).
\]
Restricting the diagonal, we obtain a function
\begin{eqnarray*}
K_{f,g,k}(x)&:=&\widetilde{K}_{f,g,k}(x,x)\\
 &=&\sum_{\alpha,\beta,\gamma} \int_{X \times X} f(y)g(z)(s_{\alpha},s_{\beta})_{k\phi_0}(y)
(s_{\beta},s_{\gamma})_{k\phi_0}(z)(s_{\gamma},s_{\alpha})_{k\phi_0}(x) \\
&\hbox{}& \hspace{18mm} MA(\phi_0)(y) \wedge MA(\phi_0)(z).
\end{eqnarray*}
\begin{theorem}[\cite{MM12}] \label{act}
We have the following asymptotic expansion:
\[
\rho_k (\phi_0) = b_{f,g,0} k^n + b_{f,g,1} k^{n-1} + \cdots
\]
for smooth functions $b_{f,g,j}$. Moreover, there are the following formula for coefficients:
\begin{eqnarray*}
b_{f,g,0}&=&fg,\\
b_{f,f,1}&=&\frac{1}{2}S_{\phi_0} f^2+2f\Delta_{\phi_0} f+\frac{1}{2}|df|_{\phi_0}^2.
\end{eqnarray*}
The expansion is uniform in $f,g$ varying in a subset of $C^{\infty}(X,{\mathbb R})$ which is compact for the $C^{\infty}$-topology. It is also uniform for $\phi_0$ as long as $\phi_0$ stays a compact set in the $C^{\infty}$-toplogy.
\end{theorem}
\section{Proof of the main theorem}
\subsection{The second variation formula for ${\mathcal D}^{(k)}$}
Before going to the proof, we define some notations that we will use later. For a hermitian matrix $A=(A_{i,j}) \in \sqrt{-1}{\mathfrak u}(N_k)$, we write $\xi_A$ for the corresponding holomorphic vector field on ${\mathbb CP}^{N_k-1}$, i.e., the push forward of $\sum_{i,j}A_{i,j}Z_i \frac{\partial}{\partial Z_j}$ via the standard projection ${\mathbb C}^{N_k} \backslash \{0\} \to {\mathbb CP}^{N_k-1}$. We set
\[
H(A):={\rm tr}(A M),
\]
then $H(A)$ is a real-valued smooth function satisfying
\[
i_{\xi_A} \omega_{FS}=\frac{\sqrt{-1}}{2 \pi} \bar{\partial}H(A),
\]
where $\omega_{FS} \in c_1({\mathcal O}(1))$ denotes the Fubini-Study metric. Moreover, if we decompose $\xi_A=\xi_A^{\mathbb R}-\sqrt{-1}J\xi_A^{\mathbb R}$, we find that $H(A)$ is the Hamiltonian for $J \xi_A^{\mathbb R}$:
\[
i_{J\xi_A^{\mathbb R}} \omega_{FS}=-\frac{1}{4 \pi}dH(A).
\]
For $A \in \sqrt{-1}{\mathfrak u}(N_k)$, let $H_{g(t)}$ be the corresponding {\it Bergman geodesic}, i.e., the family of hermitian forms corresponding to the one-parameter flow $g(t):=e^{\frac{1}{2}tA}$.
\begin{lemma} \label{fma}
The function ${\mathcal E}^{(k)}$ is affine along Bergman geodesics, i.e., we have
\[
\frac{d}{dt} {\mathcal E}^{(k)}(H_{g(t)})=\frac{1}{kN_k}{\rm tr}(A).
\]
\end{lemma}
\begin{proof}
Set $\underline{s}^{(t)}:=(s_1^{(t)}, \ldots, s_{N_k}^{(t)})=\underline{s}_0 \cdot g(t)$.
Since $(H_{g(t)}(s_i,s_j))={}^t\!g(t)^{-2}$, the direct computation shows that
\begin{eqnarray*}
\frac{d}{dt} {\mathcal E}^{(k)}(H_{g(t)})&=&-\frac{1}{kN_k} {\rm tr} \left( (H_{g(t)}(s_i,s_j))^{-1} \frac{d}{dt} (H_{g(t)}(s_i,s_j)) \right)\\
&=& -\frac{1}{kN_k}{\rm tr} \left( g(t)^2 \cdot \left( \frac{dg(t)^{-1}}{dt} \cdot g(t)^{-1} + g(t)^{-1} \cdot \frac{dg(t)^{-1}}{dt} \right) \right)\\
&=& \frac{1}{kN_k}{\rm tr} \left(g(t) \cdot \frac{1}{2}A \cdot g(t)^{-1}+\frac{1}{2}A \right)\\
&=& \frac{1}{kN_k}{\rm tr}(A).
\end{eqnarray*}
\end{proof}
\begin{lemma} \label{dfl}
\begin{enumerate}
\item we have
\[
\frac{d}{dt} {\mathcal L}(FS_k(H_{g(t)}))=k^{-1}\underline{M}(g(t)).
\]
\item
The second variation formula for ${\mathcal L} \circ FS_k$ is
\begin{eqnarray*}
\frac{d^2}{dt^2} {\mathcal L}(FS_k(H_{g(t)}))_{|t=0}&=&k^{-1} \int_X |\xi_A|_{FS}^2 \mu_{FS_k(H_0)}-k^{-2} \int_XH(A)^2 \mu_{FS_k(H_0)} \\
&+&k^{-2} \left( \int_X H(A) \mu_{FS_k(H_0)} \right)^2.
\end{eqnarray*}
\end{enumerate}
\end{lemma}
\begin{proof}
(1) Direct computation shows that
\begin{eqnarray*}
\frac{d}{dt} {\mathcal L}(FS_k(H_{g(t)}))&=&-\left( \int_X e^{-FS_k(H_{g(t)})} \right)^{-1} \cdot \int_X \left(-\frac{d}{dt}FS_k(H_{g(t)}) \right) \cdot e^{-FS_k(H_{g(t)})}\\
&=& \int_X \frac{d}{dt}FS_k(H_{g(t)}) \mu_{FS_k(H_{g(t)})},
\end{eqnarray*}
and
\begin{eqnarray} \label{dfs}
\frac{d}{dt}FS_k(H_{g(t)})&=&k^{-1} \cdot \frac{1}{\sum_i |s_i^{(t)}|^2} \cdot \frac{d}{dt} \sum_i |s_i^{(t)}|^2 \nonumber \\
&=& k^{-1} {\rm tr}((M \circ g(t)) \cdot A).
\end{eqnarray}
Thus we have
\begin{eqnarray*}
\frac{d}{dt} {\mathcal L}(FS_k(H_{g(t)})) &=& k^{-1} \int_X {\rm tr}((M \circ g(t)) \cdot A) \mu_{FS_k(H_{g(t)})}\\
&=& k^{-1}\underline{M}(g(t)).
\end{eqnarray*}
(2) We compute
\begin{eqnarray*}
\frac{d^2}{dt^2} {\mathcal L}(FS_k(H_{g(t)}))_{|t=0}&=& k^{-1} \int_X \frac{d}{dt} {\rm tr}((M \circ g(t)) \cdot A)_{|t=0} \mu_{FS_k(H_0)}\\
&+& k^{-1} \int_X H(A) \cdot \frac{d}{dt} \mu_{FS_k(H_0)}\hbox{}_{|t=0}.
\end{eqnarray*}
Since
\begin{eqnarray*}
\frac{d}{dt} {\rm tr}((M \circ g(t)) \cdot A)_{|t=0} &=& {\rm tr}(dM(\xi_A^{\mathbb R}) \cdot A)\\
&=&-4 \pi \omega_{FS}(J\xi_A^{\mathbb R}, \xi_A^{\mathbb R})\\
&=&|\xi_A|_{FS}^2,
\end{eqnarray*}
the first term is
\begin{equation} \label{ftm}
k^{-1} \int_X |\xi_A|_{FS}^2 \mu_{FS_k(H_0)}.
\end{equation}
On the other hand, using $H(A)={\rm tr}(AM)$ and \eqref{dfs}, we have
\begin{eqnarray*}
\frac{d}{dt} \mu_{FS_k(H_0)}\hbox{}_{|t=0} &=& \left( \int_X \frac{d}{dt} FS_k(H_{g(t)})_{|t=0} \mu_{FS_k(H_0)} \right) \cdot \mu_{FS_k(H_0)} \\
&-& \frac{d}{dt} FS_k(H_{g(t)})_{|t=0} \cdot \mu_{FS_k(H_0)}\\
&=& k^{-1} \left( \int_X H(A) \mu_{FS_k(H_0)}-H(A) \right) \mu_{FS_k(H_0)}.
\end{eqnarray*}
Hence the second term is
\begin{equation} \label{stm}
-k^{-2} \int_XH(A)^2 \mu_{FS_k(H_0)}
+k^{-2} \left( \int_X H(A) \mu_{FS_k(H_0)} \right)^2.
\end{equation}
Combining \eqref{ftm} and \eqref {stm} gives our conclusion.
\end{proof}
When we take into account that the Hessian $\nabla^2 {\mathcal D}^{(k)}$ is a symmetric bilinear form, we can easily get Theorem \ref{fhd} from Lemma \ref{fma} and Lemma \ref{dfl}.
\begin{proof}[Proof of Corollary \ref{coq}]
For $A \in \sqrt{-1}{\mathfrak u}(N_k)$, let $\xi_A^{\top}$ be the component of $\xi_A |_{X}$ which is tangent to $X$ and $\xi_A^{\bot}$ the component which is perpendicular to $X$ with respect to the Fubini-Study metric. Then we have
\[
i_{\xi_A^{\top}} \omega_{FS_k(H_0)} = \frac{\sqrt{-1}}{2 \pi} \bar{\partial} \left( k^{-1} H(A) \right)
\]
on $X$. It follows that $k^{-1}|\xi_A^{\top}|_{FS}^2=|\xi_A^{\top}|_{FS_k(H_0)}^2=k^{-2}|\partial H(A)|_{FS_k(H_0)}^2$. Combining with the formula $|\xi_A|_{FS}^2=|\xi_A^{\top}|_{FS}^2+|\xi_A^{\bot}|_{FS}^2$, we have
\begin{eqnarray*}
\nabla^2 {\mathcal D}^{(k)}\hbox{}_{|H_0}(A,A)&=& k^{-2} \nabla^2 {\mathcal D}\hbox{}_{|FS_k(H_0)}(H(A),H(A))+k^{-1} \int_X |\xi_A^{\bot}|_{FS}^2 \mu_{FS_k(H_0)}\\
&\geq& k^{-1} \int_X |\xi_A^{\bot}|_{FS}^2 \mu_{FS_k(H_0)} \;\;\; (\text{by Remark \ref{frk}})\\
&\geq& 0.
\end{eqnarray*}
Now we assume that $\nabla^2 {\mathcal D}^{(k)}\hbox{}_{|H_0}(A,A)=0$, then we have $\xi_A^{\bot}=0$, and hence $A \in {\rm Lie}({\rm Aut}(X,-kK_X))$ as desired. Conversely, the conditiion $A \in {\rm Lie}({\rm Aut}(X,-kK_X))$ implies that $\xi_A^{\bot}=0$ and $\xi_A^{\top}$ is holomorphic. Differentiating the equation ${\rm Ric}(\omega_{FS_k(H_0)})-\omega_{FS_k(H_0)}=\frac{\sqrt{-1}}{2 \pi} \partial \bar{\partial} h_{FS_k(H_0)}$ with respect to $\xi_A^{\top}$, we obtain
\[
-\Delta_{FS_k(H_0)} H(A)-(\partial h_{FS_k(H_0)}, \partial H(A))_{FS_k(H_0)}-H(A)+\int_X H(A) \mu_{{FS_k(H_0)}}=0.
\]
Multiplying $H(A)$ and integrating by parts, we obtain $\nabla^2 {\mathcal D}\hbox{}_{|FS_k(H_0)}(H(A),H(A))=0$. Therefore, we have $\nabla^2{\mathcal D}^{(k)}\hbox{}_{|H_0}(A,A)=0$.
\end{proof}
\subsection{Asymptotic of the Hessian $\nabla^2{\mathcal D}^{(k)}$}
Our starting point is the following:
\begin{lemma}[\cite{Fin10}, Lemma 18] \label{flm}
For any Hermitian matrices $A,B \in \sqrt{-1}{\mathfrak u}(N_k)$, we have
\[
H(A)H(B)+(\xi_A,\xi_B)_{FS}={\rm tr}(ABM).
\]
\end{lemma}
By Lemma \ref{flm}, we have
\begin{eqnarray*}
\nabla^2{\mathcal D}^{(k)}_{|H_0}(A,B)&=& k^{-1} \int_X {\rm Re} ({\rm tr}(ABM)) \mu_{FS_k(H_0)}-k^{-1}(1+k^{-1}) \int_X H(A)H(B) \mu_{FS_k(H_0)}\\
&+& k^{-2} \int_X H(A) \mu_{FS_k(H_0)} \cdot \int_X H(B) \mu_{FS_k(H_0)}.
\end{eqnarray*}
For given functions $f,g \in C^{\infty}(X, {\mathbb R})$, we set $A:=Q_{f,k}$, $B:=Q_{g,k}$ and compute the asymptotic of $\nabla^2{\mathcal D}^{(k)}_{|H_0}(Q_{f,k},Q_{g,k})$ as $k \to \infty$. However, in the course of the proof, we find that $\nabla^2{\mathcal D}^{(k)}_{|H_0}(Q_{f,k},Q_{g,k})$ has an asymptotic expansion whose coefficients are also symmetric and bilinear with respect to $f$ and $g$. Hence it follows that we may assume $f=g$ to prove Theorem \ref{coh}. We set
\begin{eqnarray*}
A_1&:=& k^{-1} \int_X {\rm tr}(Q_{f,k}^2 M) \mu_{{\mathcal T}_k(\phi_0)},\\
A_2&:=& -k^{-1}(1+k^{-1}) \int_X H(Q_{f,k})^2 \mu_{{\mathcal T}_k(\phi_0)},\\
A_3&:=& k^{-2} \left( \int_X H(Q_{f,k}) \mu_{{\mathcal T}_k(\phi_0)} \right)^2,
\end{eqnarray*}
and we will compute these terms separately, where ${\mathcal T}_k:=FS_k \circ Hilb_k$. The following arguments are based on \cite[Section 2]{Fin12}.
\begin{lemma} \label{acm}
The volume form $e^{-{\mathcal T}_k(\phi_0)}$ has the asymptotic expansion
\[
e^{-{\mathcal T}_k(\phi_0)}=(1+O(k^{-2}))e^{-\phi_0}.
\]
\end{lemma}
\begin{proof}
Since ${\mathcal T}_k(\phi_0)=\phi_0+k^{-1} \log (k^{-n} \rho_k(\phi_0))$, we have
\begin{eqnarray*}
e^{-{\mathcal T}_k(\phi_0)}&=& (k^{-n} \rho_k(\phi_0))^{-1/k} e^{-\phi}\\
&=&(1+O(k^{-2}))e^{-\phi},
\end{eqnarray*}
where we used the asymptotic expansion of $\rho_k(\phi_0)$ in the last equality (cf. Theorem \ref{abf}).
\end{proof}
\begin{lemma} \label{fsm}
The term $A_1$ has an asymptotic expansion
\[
A_1=\int_X f^2 \mu_{\phi_0} \cdot k + \int_X |\partial f|_{\phi_0}^2 \mu_{\phi_0} + O(k^{-1}).
\]
\end{lemma}
\begin{proof}
We can write $M_{|_{X}}$ as
\[
M_{|_{X}}= \left( \frac{(s_{\alpha},s_{\beta})_{k\phi_0}}{\rho_k(\phi_0)} \right)_{\alpha \beta}.
\]
It follows that
\[
A_1= k^{-1} \sum_{\alpha, \beta, \gamma} Q_{\alpha \beta} Q_{\beta \gamma} \underline{M}_{\gamma \alpha},
\]
where
\[
Q_{\alpha \beta}=\int_X(kf-\Delta_{\phi_0}f)(s_{\alpha},s_{\beta})_{k\phi_0} MA(\phi_0),
\]
\[
\underline{M}_{\gamma \alpha}=\int_X \frac{(s_{\gamma},s_{\alpha})_{k\phi_0}}{\rho_k(\phi_0)} \mu_{{\mathcal T}_k(\phi_0)}.
\]
By Theorem \ref{abf} and Lemma \ref{acm}, we have
\begin{eqnarray*}
\frac{e^{-{\mathcal T}_k(\phi_0)}}{\rho_k(\phi_0)}&=&\frac{1+O(k^{-2})}{b_0k^n+b_1k^{n-1}+O(k^{n-2})}e^{-\phi_0}\\
&=& (k^{-n}-b_1k^{-n-1}+O(k^{-n-2}))e^{-\phi_0}.
\end{eqnarray*}
It follows that
\begin{eqnarray*}
A_1&=&k^{-1} \left( \int_X e^{-{\mathcal T}_k(\phi_0)} \right)^{-1} \int_X K_k \frac{e^{-{\mathcal T}_k(\phi_0)}}{\rho_k(\phi_0)}\\
&=&k^{-1} \left( \left( \int_X e^{-\phi_0} \right)^{-1}+O(k^{-2}) \right) \int_X K_k (k^{-n}-b_1k^{-n-1}+O(k^{-n-2}))e^{-\phi_0},
\end{eqnarray*}
where we put $K_k:=K_{kf-\Delta_{\phi_0}f,kf-\Delta_{\phi_0}f,k}$. Although Theorem \ref{act} valid for functions $f,g \in C^{\infty}(X,{\mathbb R})$ which are independent of $k$, we can still apply Theorem \ref{act} to get an expansion of $K_k$ since the function $kf-\Delta_{\phi_0}f$ depends linearly on $k$. Hence we obtain
\begin{eqnarray*}
K_k&=&b_{f,f,0}k^{n+2}+(-b_{f,\Delta_{\phi_0}f,0}-b_{\Delta_{\phi_0}f,f,0}+b_{f,f,1})k^{n+1}+O(k^n)\\
&=& f^2k^{n+2}+(-2f \Delta_{\phi_0}f+b_{f,f,1})k^{n+1}+O(k^n).
\end{eqnarray*}
This gives that
\begin{eqnarray*}
A_1&=&\int_X f^2 \mu_{\phi_0} \cdot k +\int_X(-2f \Delta_{\phi_0}f+b_{f,f,1}-b_{f,f,0} \cdot b_1) \mu_{\phi_0} + O(k^{-1})\\
&=& \int_X f^2 \mu_{\phi_0} \cdot k+\int_X |\partial f|_{\phi_0}^2 \mu_{\phi_0} + O(k^{-1}).
\end{eqnarray*}
\end{proof}
\begin{lemma} \label{ahf}
There is the following expansion:
\[
H(Q_{f,k})=fk+O(k^{-1}).
\]
\end{lemma}
\begin{proof}
We write $H(Q_{f,k})$ as
\begin{eqnarray*}
H(Q_{f,k})(x)&=&\sum_{\alpha, \beta} \int_X (kf-\Delta_{\phi_0}f)(s_{\alpha},s_{\beta})_{k\phi_0}(y) (s_{\beta},s_{\alpha})_{k\phi_0}(x) \frac{1}{\rho_k(\phi_0)(x)} \cdot MA(\phi_0)(y)\\
&=& \frac{1}{\rho_k(\phi_0)(x)} K_{kf-\Delta_{\phi_0}f,k}.
\end{eqnarray*}
By Theorem \ref{afp}, we know that $K_{kf-\Delta_{\phi_0}f,k}$ has an expansion
\[
K_{kf-\Delta_{\phi_0}f,k}=b_{f,0}k^{n+1}+(-b_{\Delta_{\phi_0}f,0}+b_{f,1})k^n+O(k^{n-1}).
\]
Combining with Theorem \ref{abf}, we have
\begin{eqnarray*}
H(Q_{f,k})&=&b_{f,0}k+(-b_{\Delta_{\phi_0}f,0}+b_{f,1}-b_1 \cdot b_{f,0})+O(k^{-1})\\
&=&fk+O(k^{-1}).
\end{eqnarray*}
\end{proof}
\begin{lemma} \label{ast}
The term $A_2$ has an asymptotic expansion
\[
A_2=-\int_X f^2 \mu_{\phi_0} \cdot k-\int_X f^2 \mu_{\phi_0}+O(k^{-1}).
\]
\end{lemma}
\begin{proof}
By Lemma \ref{acm} and Lemma \ref{ahf}, we have
\begin{eqnarray*}
A_2&=&-k^{-1}(1+k^{-1}) \left( \int_X e^{-{\mathcal T}_k(\phi_0)} \right)^{-1} \cdot \int_X H(Q_{f,k})^2 e^{-{\mathcal T}_k(\phi_0)}\\
&=&-(k^{-1}+k^{-2})\left( \left( \int_X e^{-\phi_0} \right)^{-1}+O(k^{-2}) \right) \int_X (f^2k^2+O(1))(1+O(k^{-2}))e^{-\phi_0}\\
&=& -\int_X f^2 \mu_{\phi_0} \cdot k-\int_X f^2 \mu_{\phi_0}+O(k^{-1}).
\end{eqnarray*}
\end{proof}
\begin{lemma} \label{cst}
The term $A_3$ has an asymptotic expansion
\[
A_3=\left( \int_X f \mu_{\phi_0} \right)^2 +O(k^{-1}).
\]
\end{lemma}
\begin{proof}
By Lemma \ref{acm} and Lemma \ref{ahf}, we have
\begin{eqnarray*}
\int_X H(Q_{f,k}) \mu_{{\mathcal T}_k(\phi_0)}&=&\left( \left( \int_X e^{-\phi_0} \right)^{-1}+O(k^{-2}) \right) \int_X (fk+O(k^{-1}))(1+O(k^{-2}))e^{-\phi_0}\\
&=& \int_X f \mu_{\phi_0} \cdot k +O(k^{-1}).
\end{eqnarray*}
It follows that
\[
A_3=\left( \int_X f \mu_{\phi_0} \right)^2 +O(k^{-1}).
\]
\end{proof}
\begin{proof}[Proof of Theorem \ref{coh}]
Combining Lemma \ref{fsm}, Lemma \ref{ast} and Lemma \ref{cst}, we have
\begin{eqnarray*}
\nabla^2{\mathcal D}^{(k)}{}_{|H_0}(Q_{f,k},Q_{f,k})&=&A_1+A_2+A_3\\
&=& \int_X |\partial f|_{\phi_0}^2 \mu_{\phi_0}-\int_X f^2 \mu_{\phi_0}+\left( \int_X f \mu_{\phi_0} \right)^2 +O(k^{-1})\\
&=& \nabla^2 {\mathcal D}_{|\phi_0}(f,f)+O(k^{-1}).
\end{eqnarray*}
Finally, the uniformity of convergence follows from the analogous uniformity of Theorem \ref{abf}, Theorem \ref{afp} and Theorem \ref{act}.
\end{proof}
%=========References===================================================

\end{document}